\documentclass[11pt]{amsart}
\RequirePackage[top=2.5cm,bottom=2.5cm,inner=2.2cm,outer=2.2cm,
headsep=1cm,footskip=1cm]{geometry}
\usepackage{amsthm}
\newtheorem{theorem}{Theorem}[section]
\newtheorem{proposition}[theorem]{Proposition}

\newtheorem{corollary}[theorem]{Corollary}
\newtheorem{lemma}[theorem]{Lemma}
\theoremstyle{definition}
\newtheorem{definition}[theorem]{Definition}
\newtheorem{remark}[theorem]{Remark}
\RequirePackage{amsmath,amssymb,amscd}

\newcommand{\N}{\ensuremath{\mathbf{N}}}

\newcommand{\R}{\ensuremath{\mathbf{R}}}
\newcommand{\C}{\ensuremath{\mathbf{C}}}
\newcommand{\set}[1]{\ensuremath{\{ #1 \}}}
\newcommand{\suchthat}{\ensuremath{\, \vert \,}}
\newcommand{\dummy}{\ensuremath{\text{\textbullet}}}

\newcommand{\cat}[1]{\ensuremath{\mathcal{#1}}}
\newcommand{\Ob}[1]{\ensuremath{\mathrm{Ob}(#1)}}
\newcommand{\Hom}[3][]{\ensuremath{\mathrm{Hom}_{#1} (#2, #3)}}

\newcommand{\sheaf}[1]{\ensuremath{\mathcal{#1}}}
\newcommand{\germ}[2]{\ensuremath{#1_{#2}}}
\newcommand{\stalk}[2]{\ensuremath{{#1}_{#2}}}

\newcommand{\restrict}[2]{\ensuremath{#1\vert_{#2}}}
\newcommand{\struct}[1]{\ensuremath{\mathcal{O}_{#1}}}
\newcommand{\ringed}[1]{\ensuremath{(#1,\struct{#1})}}
\newcommand{\rk}[2][]{\ensuremath{{\mathrm{rk}}_{#1}(#2)}}
\newcommand{\contstruct}[1]{\ensuremath{\mathcal{C}_{#1}}}

\newcommand{\contralg}[1]{\ensuremath{C(#1)}}
\newcommand{\cinfty}[1]{\ensuremath{\mathcal{C}^{\infty}(#1)}}
\newcommand{\Stein}[1]{\ensuremath{\mathfrak{#1}}}
\newcommand{\sstalk}[2]{\ensuremath{\struct{#1,#2}}}
\newcommand{\nsstalk}[3][]{\ensuremath{\struct{#2,#3}^{#1}}}
\newcommand{\stion}[2]{\ensuremath{\Gamma({#2},{#1})}}  
\newcommand{\mor}[3]{\ensuremath{{#1}:{#2}\rightarrow{#3}}}
\newcommand{\Cohgp}[3][]{\ensuremath{{\mathrm{H}}^{#1}({#2},{#3})}}

\newcommand{\Shom}[3][]{\ensuremath{{\mathcal{H}om}_{#1}({#2},{#3})}} 

\newcommand{\id}[1]{\ensuremath{\mathbf{1}_{#1}}}
\newcommand{\dsum}{\ensuremath{\oplus}}
\newcommand{\iso}{\ensuremath{\cong}}

\newcommand{\Ker}[1]{\ensuremath{\mathrm{Ker} (#1)}}

\newcommand{\modcat}[1]{\ensuremath{{#1}\text{-}\mathbf{mod}}}
\newcommand{\Lfb}[1]{\ensuremath{\mathbf{Lfb}({#1})}}

\newcommand{\Qcoh}[1]{\ensuremath{\mathbf{Qcoh}({#1})}}
\newcommand{\Coh}[1]{\ensuremath{\mathbf{Coh}(#1)}}
\newcommand{\Fgp}[1]{\ensuremath{\mathbf{Fgp}({#1})}} 
\newcommand{\Tensor}[3][]{\ensuremath{#2\otimes_{#1} #3}}

\newcommand{\Presheaf}[1]{\ensuremath{\mathcal{P}({#1})}} 

\newcommand{\Ext}[4][]{\ensuremath{{\mathrm{Ext}_{#2}^{#1}}(#3,#4)}} 

\newcommand{\X}[1]{\ensuremath{\mathrm{D} (#1)}}
\newcommand{\loc}[2]{\ensuremath{{#1}_{#2}}}

\newcommand{\adiffspace}[1]{\ensuremath{(\rspec{#1},\tilde{#1})}} 
\newcommand{\rspec}[1]{\ensuremath{\mathrm{Spec}_{\mathrm{r}} (#1)}}
\RequirePackage{url}
\newcommand{\arXivURL}{http://www.arXiv.org/abs/}

\newcommand{\biburl}[1]{\newline URL \url{#1}}

\renewcommand{\phi}{\ensuremath{\varphi}}
\renewcommand{\epsilon}{\ensuremath{\varepsilon}}

\RequirePackage{xr-hyper}
\RequirePackage{hyperref}
\hypersetup{bookmarks=true}%
\hypersetup{bookmarksopen=true}%

\begin{document}
\title{Note on the Serre-Swan Theorem}
\author{Archana~S.~Morye}
\address{Harish-Chandra Research Institute, Chhatnag Road, Jhusi, Allahabad
211 019, India}
\email{sarchana@mri.ernet.in}
\subjclass[2000]{Primary: 14F05; Secondary: 16D40, 18F20.}
\keywords{Serre-Swan Theorem, vector bundles, projective modules}  
\begin{abstract}
  \noindent We determine a class of ringed spaces, $\ringed{X}$ for which the
  category of locally free sheaves of bounded rank over $X$ is equivalent to
  the category of finitely generated projective
  $\stion{\struct{X}}{X}$-modules.  The well-known Serre-Swan theorems for
  affine schemes, differentiable manifolds, Stein spaces, etc., are then
  derived.
\end{abstract}
\maketitle
\section{Introduction}
\label{sec:Intro}
It is a well-known theorem due to Serre, that for an affine scheme
$\ringed{X}$, there is an equivalence of categories between locally free
$\struct{X}$-modules of finite rank, and finitely generated projective
$\stion{\struct{X}}{X}$-modules \cite[Section 50, Corollaire to Proposition 4,
p.~242]{Ser55}.  Later Swan proved the same equivalence when $X$ is a
paracompact topological space of finite covering dimension, and $\struct{X}$
is the sheaf of continuous real-valued functions on $X$ \cite[Theorem 2 and
p.~277]{Sw62}.  In this article we will generalize this result to a large
class of locally ringed spaces.

For any ringed space $\ringed{X}$, let $\modcat{\struct{X}}$ denote the
category of $\struct{X}$-modules, and $\Lfb{X}$ the full subcategory of
$\modcat{\struct{X}}$ consisting of locally free $\struct{X}$-modules of
bounded rank.  For any ring $A$, let $\modcat{A}$ denote the category of
$A$-modules, and $\Fgp{A}$ the full subcategory of $\modcat{A}$ consisting of
finitely generated projective $A$-modules.  We will say that the Serre-Swan
Theorem holds for a ringed space $\ringed{X}$ if the canonical functor
$\stion{\dummy}{X}:\Lfb{X} \rightarrow \Fgp{\stion{\struct{X}}{X}}$ is an
equivalence of categories.

Recall that an $\struct{X}$-module $\sheaf{F}$ is said to be {\it generated by
  global sections} if there is a family of sections $(s_i)_{i\in I}$ in
$\stion{\sheaf{F}}{X}$ such that for each $x\in X$, the images of $s_i$ in the
stalk $\stalk{\sheaf{F}}{x}$ generate that stalk as an $\sstalk{X}{x}$-module.
We will say that $\sheaf{F}$ is {\it finitely generated by global sections} if
a finite family of global sections $(s_i)_{i\in I}$ exists with the above
property.

\begin{definition} \label{def:1} Let $\ringed{X}$ be a locally ringed
  space.  Then, a subcategory $\cat{C}$ of $\modcat{\struct{X}}$ is
  called an {\it admissible subcategory} if it satisfies the following
  conditions:
  \begin{enumerate}\label{conditions}
  \item[{\bf C1.}] $\cat{C}$ is a full abelian subcategory of
    $\modcat{\struct{X}}$, and
    $\Shom[\struct{X}]{\sheaf{F}}{\sheaf{G}}$ belongs to $\cat{C}$ for
    every pair of sheaves $\sheaf{F}$ and $\sheaf{G}$ in $\cat{C}$,
    where $\Shom[\struct{X}]{\sheaf{F}}{\sheaf{G}}$ denotes the sheaf
    of $\struct{X}$-morphisms from $\sheaf{F}$ to $\sheaf{G}$.
  \item[{\bf C2.}] Every sheaf in $\cat{C}$ is acyclic, and generated by
    global sections.
  \item[{\bf C3.}] $\Lfb{X}$ is a subcategory of $\cat{C}$.
  \end{enumerate}
\end{definition}

In Section 2, we will prove that the Serre-Swan Theorem holds for a locally
ringed space $\ringed{X}$ if every locally free $\struct{X}$-module of bounded
rank is finitely generated by global sections, and if the category
$\modcat{\struct{X}}$ contains an admissible subcategory.  In Section 3, we
will show that certain wellknown classes of locally ringed spaces satisfy the
above conditions, and hence the Serre-Swan Theorem holds for them.

\section{The Serre-Swan Theorem}
\label{sec:serre-swan-theorem}

We will assume that all rings are commutative unless otherwise mentioned.  All
compact and paracompact topological spaces are assumed to be Hausdorff.  Let
$\ringed{X}$ be a ringed space, and let $A$ denote the ring
$\stion{\struct{X}}{X}$.  We will denote the $\stion{\struct{X}}{X}$-module of
homomorphism of $\struct{X}$-modules $\sheaf{F}$ and $\sheaf{G}$ by
$\Hom[\struct{X}]{\sheaf{F}}{\sheaf{G}}$.  If $\phi \in
\Hom[\struct{X}]{\sheaf{F}}{\sheaf{G}}$, and $U$ is an open subset of $X$,
then $\mor{\phi_{U}}{\sheaf{F}(U)}{\sheaf{G}(U)}$ denotes the homomorphism of
$\struct{X}(U)$-modules induced by $\phi$.

We have a canonical functor $\stion{\dummy}{X}:\modcat{\struct{X}} \rightarrow
\modcat{A}$.  Given an $A$-module $M$, we get a presheaf $\Presheaf{M}$ on $X$
by defining $\Presheaf{M}(U)=\Tensor[A]{M}{\struct{X}(U)}$ for every open set
$U$ of $X$.  We will denote by $\sheaf{S}(M)$ the sheaf associated to the
presheaf $\Presheaf{M}$.  Similarly, for a homomorphism $\mor{u}{M}{N}$ of
$A$-modules, we get a homomorphism
$\mor{\sheaf{S}(u)}{\sheaf{S}(M)}{\sheaf{S}(N)}$ of $\struct{X}$-modules.
Thus, we get a functor $\sheaf{S} :\modcat{A}\rightarrow \modcat{\struct{X}}$.
Since, for every $x$ in $X$, the functor $\Tensor[A]{\dummy}{\sstalk{X}{x}}$
is right exact, the functor $\sheaf{S}$ is right exact.

Note that the functor $\sheaf{S}$ is a left adjoint of $\stion{\dummy}{X}$.
Indeed, let $\sheaf{F}$ be an $\struct{X}$-module, and $M$ an $A$-module.
Consider a homomorphism $\mor{u}{M}{\stion{\sheaf{F}}{X}}$ of $A$-modules.
Let $\mor{v'}{\Presheaf{M}}{\sheaf{F}}$ be the morphism of presheaves such
that for every open subset $U$ of $X$,
$\mor{v'_U}{\Tensor[A]{M}{\struct{X}(U)}}{\sheaf{F}(U)}$ is given by
$v'_U(\Tensor[A]{m}{f})=f\cdot \restrict{u(m)}{U}$ for $m\in M$ and $f\in
\struct{X}(U)$.  Let $\mor{v_u}{\sheaf{S}(M)}{\sheaf{F}}$ be the morphism of
sheaves associated to $v'$. Define
\begin{equation}\label{eq:1}
  \mor{\lambda_{\sheaf{F},M}}{\Hom[A]{M}{\stion{\sheaf{F}}{X}}}{\Hom[\struct{X}]{\sheaf{S}(M)}{\sheaf{F}}}, \quad \lambda_{\sheaf{F},M}(u)=v_u.
\end{equation}
It is easy to see that the map $\lambda_{\sheaf{F},M}$ is a functorial
isomorphism.

\begin{theorem} \label{SS Theorem} Let $\ringed{X}$ be a locally
  ringed space, and let $A=\stion{\struct{X}}{X}$.  Assume that
  $\modcat{\struct{X}}$ contains an admissible subcategory $\cat{C}$,
  and that every sheaf in $\Lfb{X}$ is finitely generated by global
  sections.  Then, $\mor{\stion{\dummy}{X}}{\Lfb{X}}{\Fgp{A}}$ is an
  equivalence of categories, i.e., the Serre-Swan Theorem holds for
  $\ringed{X}$.
\end{theorem}
To prove the above theorem we will need some preliminary results.

\begin{proposition} \label{prop:1} Let $\ringed{X}$ be a ringed space,
  and let $\cat{C}$ be a full abelian subcategory of
  $\modcat{\struct{X}}$, such that $\struct{X}$ belongs to $\cat{C}$.
  Suppose that every sheaf in $\cat{C}$ is generated by global
  sections.  Then $\stion{\dummy}{X}:\cat{C} \rightarrow
  \modcat{\stion{\struct{X}}{X}}$ is fully faithful.
\end{proposition}
\begin{proof} \label{proof:1} Let $A=\stion{\struct{X}}{X}$, and let
  $\sheaf{F}, \sheaf{G} \in \Ob{\cat{C}}$.  Then, we have to show that the
  homomorphism
  \begin{equation}\label{eq:2}
    \phi_{\sheaf{F},\sheaf{G}}:\Hom[\struct{X}]{\sheaf{F}}{\sheaf{G}}\rightarrow
    \Hom[A]{\stion{\sheaf{F}}{X}}{\stion{\sheaf{G}}{X}},\quad u \mapsto u_{X}
    \end{equation}
    is a bijection.  First we will prove that $\phi_{\sheaf{F},\sheaf{G}}$ is
    injective.  Let $u \in \Hom[\struct{X}]{\sheaf{F}}{\sheaf{G}}$ be such
    that $u_X=0$, we will then show that $u=0$. To show this it suffices to
    show that $u_x(w)=0$ for all $x\in X$, and $w \in \stalk{\sheaf{F}}{x}$.
    Since $\sheaf{F}$ is generated by global sections, $w=\sum_{i=1}^{n} a_i
    \cdot \germ{(s_i)}{x}$, where $a_i \in \sstalk{X}{x}$, and $s_i \in
    \stion{\sheaf{F}}{X}$, $i=1,\ldots,n$.  Thus,
    \begin{equation*} \label{eq:3} u_x(w) =u_x\left(\sum_{i=1}^{n} a_i \cdot
        \germ{(s_i)}{x}\right)=\sum_{i=1}^{n} a_i\cdot
      u_x\left(\germ{(s_i)}{x}\right) =\sum_{i=1}^{n} a_i \cdot
      \germ{\left(u_X(s_i)\right)}{x}=\sum_{i=1}^{n} a_i \cdot 0=0.
  \end{equation*}
  Hence we are through.  Now we will show that $\phi_{\sheaf{F},\sheaf{G}}$ is
  surjective.  Let $\mor{\alpha}{\stion{\sheaf{F}}{X}}{\stion{\sheaf{G}}{X}}$
  be a homomorphism of $A$-modules.  For every point $x$ in $X$, let us define
  $\mor{\alpha^x}{\stalk{\sheaf{F}}{x}}{\stalk{\sheaf{G}}{x}}$ as follows.
  Let $w=\sum_{i=1}^{n} a_i \cdot \germ{(s_i)}{x}$ in $\stalk{\sheaf{F}}{x}$
  defined as above.  Define $\alpha^{x}(w)=\sum_{i=1}^{n}a_i \cdot
  \germ{(\alpha(s_i))}{x}$.  To prove $\alpha^x$ is welldefined it is enough
  to check that if $\sum_{i=1}^{n}a_i\cdot \germ{(s_i)}{x}=0$, then
  $\sum_{i=1}^{n}a_i \cdot \germ{(\alpha(s_i))}{x}=0$.  Consider a
  homomorphism of $\struct{X}$-modules $\mor{\phi}{\struct{X}^n}{\sheaf{F}}$
  defined by the family $(s_i)_{i=1}^{n}$.  Then, $\sheaf{K}=\Ker{\phi} \in
  \Ob{\cat{C}}$, therefore $\sheaf{K}$ is also generated by global sections.
  As $(a_1,\ldots,a_n) \in \stalk{\sheaf{K}}{x}$, there exist $g_1,\ldots,g_p
  \in \stion{\sheaf{K}}{X}\subseteq \struct{X}^n(X) = A^n$, and
  $z_1,\ldots,z_p \in \sstalk{X}{x}$ such that $(a_1,\ldots,a_n)=\sum_{j=1}^p
  z_j \cdot \germ{(g_j)}{x}$.  Let $g_j=(g_{j1},\ldots,g_{jn})$, for $g_{ji}
  \in A$, $i=1,\ldots,n$, and $j=1,\ldots,p$.  Then $a_i=\sum_{j=1}^p z_j
  \cdot \germ{(g_{ji})}{x}$, for $i=1,\ldots,n$, and $\sum_{i=1}^n
  g_{ji}s_i=0$, for $j=1,\ldots,p$.  Consider,
  \begin{equation*} \label{eq:4} 
    \sum_{i=1}^n a_i \cdot \germ{\alpha(s_i)}{x}
    =\sum_{i,j=1}^{n,p} z_j \cdot \germ{(g_{ji})}{x} \germ{\alpha(s_i)}{x}
    =\sum_{j=1}^p z_j \cdot \germ{\left(\sum_{i=1}^ng_{ji}\alpha(s_i)\right)}{x}
    =\sum_{j=1}^p z_j \cdot\germ{\left(\alpha\left(\sum_{i=1}^n
          g_{ji}s_i\right)\right)}{x}=0.
  \end{equation*}
  Therefore $\alpha^x$ is a welldefined map for every $x$ in $X$.  Now, we
  will check that the $\alpha^x,(x\in X)$ give rise to a homomorphism of
  $\struct{X}$-module $\mor{u}{\sheaf{F}}{\sheaf{G}}$.  Let $U$ be an open
  neighborhood of $x$, and let $s\in \sheaf{F}(U)$.  We have as before, $s_x
  =\sum_{i=1}^{n} a_i \cdot \germ{(s_i)}{x}$ for some $s_i \in
  \stion{\sheaf{F}}{X}$, and for some $a_i \in \sstalk{X}{x}$, $i=1,\ldots,n$.
  Thus, there exist an open neighborhood $V$ of $x$ in $U$, and
  $f_i\in\struct{X}(V)$, such that $a_i=\germ{(f_i)}{x}$, $(i=1,\ldots,n$),
  and $\restrict{s}{V}=\sum_{i=1}^{n} f_i \cdot \restrict{s_i}{V}$.  Define
  $t= \sum_{i=1}^{n} f_i \cdot \restrict{\alpha(s_i)}{V}\in \sheaf{G}(V)$.
  Then, $\alpha^y(\germ{s}{y})=\alpha^y \left(\sum_{i=1}^n \germ{(f_i)}{y}
    \germ{(s_i)}{y}\right)=\sum_{i=1}^n
  \germ{(f_i)}{y}\germ{(\alpha(s_i))}{y}=\germ{t}{y}$, for all $y$ in $V$.
  This proves that, there exists a unique homomorphism
  $\mor{u}{\sheaf{F}}{\sheaf{G}}$ of $\struct{X}$-modules such that
  $\stalk{u}{x} = \alpha^x$ for all $x$ in $X$.  Also it is easy to check that
  $u_X=\alpha$.
\end{proof}
\begin{lemma} \label{lemma:1} Let $\ringed{X}$ be a ringed space, such that
  $X$ is a paracompact topological space, and $\struct{X}$ is a fine sheaf.
  Consider an $\struct{X}$-module $\sheaf{F}$.  Let $x$ be a point in $X$, $U$
  an open neighborhood of $x$, and $s'\in \sheaf{F}(U)$.  Then, there exist an
  open neighborhood $V$ of $x$ in $U$, and a global section $s$ of
  $\sheaf{F}$, such that $\restrict{s}{V}=\restrict{s'}{V}$.  In particular,
  the canonical homomorphism
  $\mor{\rho_x}{\stion{\sheaf{F}}{X}}{\stalk{\sheaf{F}}{x}}$ is surjective,
  and hence $\sheaf{F}$ is generated by global sections.
\end{lemma}
The above lemma is standard when $\ringed{X}$ is a differential manifold, and
the proof in the general case is similar.  It follows from Proposition
\ref{prop:1} that for a ringed space $X$ satisfying the conditions of Lemma
\ref{lemma:1}, the functor
$\mor{\stion{\dummy}{X}}{\modcat{\struct{X}}}{\modcat{A}}$ is fully faithful.

\begin{remark} \label{remark:1} Let $\ringed{X}$ be a ringed space, and let
  $A$ denote the ring $\stion{\struct{X}}{X}$.  If $\cat{C}$ is a subcategory
  of $\modcat{\struct{X}}$, define ${\modcat{A}}_{\cat{C}}$ to be the full
  subcategory of $\modcat{A}$ consisting of $A$-modules $M$ such that $M\iso
  \stion{\sheaf{F}}{X}$ for some $\sheaf{F}$ in $\cat{C}$.  Let $\cat{C}$ be a
  subcategory of $\modcat{\struct{X}}$ as in Proposition \ref{prop:1}, then
  $\mor{\stion{\dummy}{X}}{\cat{C}}{{\modcat{A}}_{\cat{C}}}$ is an equivalence
  of categories.  Indeed, it follows from Proposition \ref{prop:1} that
  $\stion{\dummy}{X}$ is fully faithful, and by the definition of
  ${\modcat{A}}_{\cat{C}}$, it is essentially surjective.
\end{remark}

\begin{proposition} \label{prop:2}
  Let $\ringed{X}$ be a ringed space, and let $A$ denote the ring
  $\stion{\struct{X}}{X}$.  Suppose $\cat{C}$ is as in Proposition
  \ref{prop:1}.  Then, $\sheaf{F}$ is isomorphic to
  $\sheaf{S}(\stion{\sheaf{F}}{X})$ for every sheaf $\sheaf{F}$ in $\cat{C}$ .
\end{proposition}
\begin{proof} \label{proof:2} Let $M=\stion{\sheaf{F}}{X}$.  Recall that
  $\sheaf{P}(M)$ denotes the presheaf tensor product of $M$ and $\struct{X}$
  over $A$.  Let $\mor{u'}{\sheaf{P}(M)}{\sheaf{F}}$ be the morphism of
  presheaves such that for every open subset $U$ of $X$,
  $\mor{{u'}_U}{\Tensor[A]{M}{\struct{X}(U)}}{\sheaf{F}(U)}$ is given by
  ${u'}_U(\Tensor[A]{m}{f})=f\cdot\restrict{m}{U}$ for $m\in M$ and
  $f\in\struct{X}(U)$.  Let $\mor{u}{\sheaf{S}(M)}{\sheaf{F}}$ be the morphism
  of sheaves associated to $u'$.  (Note that,
  $u=\lambda_{\sheaf{F},M}(\id{M})$, where $\lambda_{\sheaf{F},M}$ is
  (\ref{eq:1}), i.e., $u$ is the counit morphism of $\sheaf{F}$ with respect
  to the adjunction $\lambda$.)  Thus,
  $\mor{u_x}{\Tensor[A]{M}{\sstalk{X}{x}}}{\stalk{\sheaf{F}}{x}}$ is such that
  $\stalk{u}{x}(\sum_{i=1}^n \Tensor[A]{s_i}{a_i})=\sum_{i=1}^n a_i \cdot
  \germ{(s_i)}{x}$, for $x\in X$, $a_i\in \sstalk{X}{x}$, and $s_i\in M$,
  $i=1,\ldots,n$.  Since $u$ is a morphism of sheaves to prove that $u$ is an
  isomorphism, it is enough to prove that $\stalk{u}{x}$ is an isomorphism for
  every $x$ in $X$.  The sheaf $\sheaf{F}$ is generated by global sections,
  therefore any germ $w$ belongs to $\stalk{\sheaf{F}}{x}$ can be written as
  $w=\sum_{i=1}^n a_i \cdot \germ{(s_i)}{x}$, for some $a_i\in\sstalk{X}{x}$,
  and $s_i \in M$, $i=1,\ldots,n$.  Define $\alpha=\sum_{i=1}^n
  \Tensor[A]{s_i}{a_i}$.  Then $
  \stalk{u}{x}(\alpha)=\stalk{u}{x}(\sum_{i=1}^n \Tensor[A]{s_i}
  {a_i})=\sum_{i=1}^n a_i \cdot \germ{(s_i)}{x}=w$.  Therefore, $\stalk{u}{x}$
  is surjective.  Now we will show that $\stalk{u}{x}$ is injective.  Let
  $w=\sum_{i=1}^n \Tensor[A]{s_i}{a_i}$, where $s_i\in M$, and $a_i\in
  \sstalk{X}{x}$, $i=1,\ldots, n$, and suppose that
  $\stalk{u}{x}(w)=\sum_{i=1}^n a_i \cdot \germ{(s_i)}{x}=0$.  Consider the
  morphism $\mor{\phi}{\struct{X}^n}{\sheaf{F}}$ of $\struct{X}$-modules
  defined by the family $(s_i)_{i=1}^{n}$.  Then, $\sheaf{K}=\Ker{\phi}$
  belongs to $\cat{C}$, and $(a_1,\ldots,a_n)$ belongs to
  $\stalk{\sheaf{K}}{x}$.  Therefore, there exist $g_1,\ldots,g_p \in
  \stion{\sheaf{K}}{X} \subseteq \struct{X}^n(X) = A^n$, and $z_1,\ldots,z_p
  \in \sstalk{X}{x}$ such that $(a_1,\ldots,a_n)=\sum_{j=1}^p z_j \cdot
  \germ{(g_j)}{x}$.  Let $g_j=(g_{j1},\ldots,g_{jn})$, where $g_{ji} \in A$,
  $i=1,\ldots,n$, and $j=1,\ldots,p$.  Then, $a_i=\sum_{j=1}^p z_j \cdot
  \germ{(g_{ji})}{x}$ for $i=1,\ldots,n$, and $\sum_{i=1}^n g_{ji}s_i=0$, for
  $j=1,\ldots,p$.  Thus,
  \begin{equation*}\label{eq:5}
    w=\sum_{i=1}^n \Tensor[A]{s_i}{a_i} =\sum_{i,j=1}^{n,p}
    \Tensor[A]{s_i}{z_j \cdot \germ{(g_{ji})}{x}} =\sum_{i,j=1}^{n,p}
    \Tensor[A]{\left(g_{ji}s_i\right)}{z_j} =\sum_{j=1}^p
    \Tensor[A]{\left(\sum_{i=1}^n g_{ji}s_i \right)}{z_j} =0.
  \end{equation*}  
  Therefore $\stalk{u}{x}$ is injective.  
\end{proof}
\begin{remark}\label{remark:2}
  One can also prove Proposition \ref{prop:1} using Proposition \ref{prop:2}.
  Indeed we can define the inverse of a morphism $\phi_{\sheaf{F},\sheaf{G}}$,
  (\ref{eq:2}).  If $v$ is the inverse of $u$ in Proposition \ref{prop:2},
  then for all $x$ in $X$,
  $\mor{\stalk{v}{x}}{\stalk{\sheaf{F}}{x}}{\Tensor[A]{\stion{\sheaf{F}}{X}}{\sstalk{X}{x}}},
  w\mapsto \sum_{i=1}^n \Tensor[A]{s_i}{a_i}$, where $w=\sum_{i=1}^n a_i \cdot
  \germ{(s_i)}{x}$, $a_i\in \sstalk{X}{x}$, $s_i\in \stion{\sheaf{F}}{X}$,
  $i=1,\ldots,n$.  Let $u_{\sheaf{G}}$ be the counit morphism of $\sheaf{G}$
  with respect to the adjunction $\lambda$, (\ref{eq:1}).  Then, the inverse
  of $\phi_{\sheaf{F},\sheaf{G}}$ is given by
  $\mor{\psi}{\Hom[A]{\stion{\sheaf{F}}{X}}{\stion{\sheaf{G}}{X}}}{\Hom[\struct{X}]{\sheaf{F}}{\sheaf{G}}}$,
  $\alpha \mapsto u_{\sheaf{G}} \circ \sheaf{S}(\alpha) \circ v $.
\end{remark}
\begin{corollary}\label{cor:1}
  Let $\ringed{X}$ be a locally ringed space, and let $A$ denote its ring of
  global sections.  Assume that $\modcat{\struct{X}}$ contains an admissible
  subcategory $\cat{C}$.  Then, every finitely generated projective module $P$
  is isomorphic to $\stion{\sheaf{S}(P)}{X}$.
\end{corollary}
\begin{proof}\label{proof:3} 
  Every finitely generated projective module is finitely presented, therefore
  we get an exact sequence, $A^p \stackrel{\alpha}{\rightarrow } A^q
  \rightarrow P \rightarrow 0$, for some $p,q \in \N$.  Applying the functor
  $\sheaf{S}$, and by Proposition \ref{prop:2} we get an exact sequence
  $\struct{X}^p \stackrel{\psi(\alpha)}{\rightarrow }\struct{X}^q \rightarrow
  \sheaf{S}(P) \rightarrow 0$, where $\psi$ is as in Remark \ref{remark:2}.
  Since $\cat{C}$ is an admissible subcategory by {\bf C3} sheaf
  $\struct{X}^m$ is in $\cat{C}$ for every integer $m$, and so by {\bf C1} the
  sheaf $\sheaf{S}(P)\in\Ob{\cat{C}}$.  Also $\stion{\dummy}{X}$ is an exact
  functor restricted to the category of acyclic $\struct{X}$-modules, it
  follows that $A^p \stackrel{\psi(\alpha)_{X}}{\longrightarrow } A^q
  \rightarrow \stion{\sheaf{S}(P)}{X} \rightarrow 0$ is an exact sequence.  By
  Remark \ref{remark:2}, $\psi(\alpha)_{X}=\alpha$.  Therefore
  $\stion{\sheaf{S}(P)}{X} \iso P$, being cokernels of the same map.
\end{proof}

\begin{proposition} \label{prop:3} Let $\ringed{X}$ be a locally ringed space,
  and let $A$ denote the ring $\stion{\struct{X}}{X}$.  Assume that
  $\modcat{\struct{X}}$ contains an admissible subcategory $\cat{C}$.  If a
  locally free sheaf of bounded rank $\sheaf{F}$ is finitely generated by
  global sections, then $\stion{\sheaf{F}}{X}$ is a finitely generated
  projective $A$-module.
\end{proposition}
\begin{proof} \label{proof:4} 
  By the hypothesis, there exists a surjective morphism
  $\mor{u}{\struct{X}^n}{\sheaf{F}}$ for some $n\in\N$.  Consider the exact
  sequence of $\struct{X}$-modules
  \begin{equation} \label{exactseq:1} 
    0\rightarrow \sheaf{K} \rightarrow
    \struct{X}^n \stackrel{u}{\rightarrow}\sheaf{F} \rightarrow 0
  \end{equation}
  where $\sheaf{K}=\ker{u}$.  Since $\sheaf{F}$ is locally free of bounded
  rank, by \cite[Corollaire to Proposition 4.2.3, p.~189]{Gro57}
  $\Ext[1]{\struct{X}}{\sheaf{F}}{\sheaf{K}} \iso
  \Cohgp[1]{X}{\Shom[\struct{X}]{\sheaf{F}}{\sheaf{K}}}$.  As $\cat{C}$ is an
  admissible subcategory $\Shom[\struct{X}]{\sheaf{F}}{\sheaf{K}} \in
  \Ob{\cat{C}}$.  Hence $\Shom[\struct{X}]{\sheaf{F}}{\sheaf{K}}$ is acyclic.
  Therefore $\Ext[1]{\struct{X}}{\sheaf{F}}{\sheaf{K}}=0$, and consequently
  the exact sequence (\ref{exactseq:1}) splits.  This implies that
  $\struct{X}^n\iso\sheaf{F}\dsum \sheaf{K}$.  Thus $A^n \iso
  \stion{\struct{X}^n}{X} \iso \stion{\sheaf{F}}{X} \dsum
  \stion{\sheaf{K}}{X}$.  This proves that $\stion{\sheaf{F}}{X}$ is a
  finitely generated projective $A$-module.
\end{proof}
\begin{remark}\label{remark:3}
  Let $\ringed{X}$ is a ringed space.  Recall that an $\struct{X}$-module
  $\sheaf{F}$ is said to be of {\it finite type} if for every point $x\in X$,
  there exists an open neighborhood $U$ of $x$ such that
  $\restrict{\sheaf{F}}{U}$ is generated by a finite family of sections of
  $\sheaf{F}$ on $U$ \cite[Chap. 0, 5.2.1, p.~45]{Gro60}.  An
  $\struct{X}$-module $\sheaf{F}$ is said to be is of {\it finite
    presentation} if for every point $x$ in $X$, there exist an open
  neighborhood $U$ of $x$, and an exact sequence of
  $\restrict{\struct{X}}{U}$-modules
  $\restrict{\struct{X}^p}{U}\stackrel{u}{\rightarrow}
  \restrict{\struct{X}^q}{U}\stackrel{v}{\rightarrow} \restrict{\sheaf{F}}{U}
  \rightarrow 0$ \cite[Chap. 0, 5.2.5, p.~46]{Gro60}.  Let $\sheaf{F}$ and
  $\sheaf{G}$ be $\struct{X}$-modules of finite presentation.  Let $x \in X$,
  and suppose that the $\sstalk{X}{x}$-modules $\stalk{\sheaf{F}}{x}$ and
  $\stalk{\sheaf{G}}{x}$ are isomorphic.  Then, there exists an open
  neighborhood $U$ of $x$ such that the $\restrict{\struct{X}}{U}$-modules
  $\restrict{\sheaf{F}}{U}$ and $\restrict{\sheaf{G}}{U}$ are isomorphic
  \cite[Chap. 0, 5.2.7, pp.~46-47]{Gro60}.
\end{remark}

\begin{lemma} \label{lemma:2} Let $\ringed{X}$ be a ringed space, and let
  $\sheaf{F}$ be an $\struct{X}$-module of finite presentation.  Let $x\in X$,
  and suppose that $\stalk{\sheaf{F}}{x}$ is a free $\sstalk{X}{x}$-module.
  Then there exist an open neighborhood $U$ of $x$, and $n\in \N$, such that
  $\restrict{\sheaf{F}}{U}$ is isomorphic to $\restrict{\struct{X}^n}{U}$.
\end{lemma}
\begin{proof}\label{proof:5}
  Since $\stalk{\sheaf{F}}{x}$ is a free $\sstalk{X}{x}$-module, there exists
  an integer $n\in \N$ such that $\sstalk{X}{x}$-modules $\nsstalk[n]{X}{x}$
  and $\stalk{\sheaf{F}}{x}$ are isomorphic.  All locally free sheaves of
  finite type are of finite presentation, hence $\struct{X}^n$ is of finite
  presentation.  From Remark \ref{remark:3} there exists an open neighborhood
  $U$ of $x$ such that $\restrict{\sheaf{F}}{U}\iso
  \restrict{\struct{X}^{n}}{U}$.
\end{proof}

\begin{remark}\label{remark:4}
  One can prove Lemma \ref{lemma:2} using Fitting ideals. Indeed, let $M$ be a
  finitely generated $A$-module.  For $r \in \N$ we denote the $r$-th Fitting
  ideal of $M$ by $\mathfrak{F}_r(M)$. Let $\ringed{X}$ be a ringed space, and
  consider an $\struct{X}$-module $\sheaf{G}$ of finite presentation.  We
  denote the $r$-th Fitting ideal sheaf of $\sheaf{G}$ by
  $\sheaf{F}_r(\sheaf{G})$, for $r \geq 0$.  For every $x \in X$,
  $\stalk{(\sheaf{F}_r(\sheaf{G}))}{x}=\mathfrak{F}_r(\stalk{\sheaf{G}}{x})$,
  and the ideal sheaf $\sheaf{F}_r(\sheaf{G})$ is of finite type as an
  $\struct{X}$-module.  Since $\sheaf{F}_r(\sheaf{G})$ is of finite type, if
  $\stalk{(\sheaf{F}_r(\sheaf{G}))}{x} = \sstalk{X}{x}$, then there exists an
  open neighborhood $U$ of $x$ such that $\stalk{(\sheaf{F}_r(\sheaf{G}))}{y}
  = \sstalk{X}{y}$, for all $y$ in $U$.  Also, the support of
  $\sheaf{F}_r(\sheaf{G})$, that is the set $\set{x\in X \suchthat
    \stalk{(\sheaf{F}_r(\sheaf{G}))}{x}\neq 0}$ is closed in $X$.  Further let
  $x\in X$, and an $\sstalk{X}{x}$-module $\stalk{\sheaf{G}}{x}$ is free of
  rank $n$.  Then there exists an open neighborhood $U$ of $x$ such that for
  all $y \in U$, $\stalk{\sheaf{G}}{y}$ is a free $\sstalk{X}{y}$-module of
  rank $n$.  (This follows from the fact that, if $M$ is a free $A$-module of
  rank $q$, then $\mathfrak{F}_r(M)=0$ for $0\leq r < q$, and
  $\mathfrak{F}_r(M)=A$ for $r\geq q$ \cite[Exercise 1, p.~90]{Nor76}, and
  converse is true if $A$ is a local ring \cite[Propostion 20.8,
  p.~500]{Eis95}.)  Now since $\sheaf{G}$ is of finite type
  $\restrict{\sheaf{G}}{U}\iso \restrict{\struct{X}^{n}}{U}$.
\end{remark}

\begin{lemma} \label{lemma:3} Let $\ringed{X}$ be a locally ringed space, and
  let $A$ denote the ring $\stion{\struct{X}}{X}$.  Suppose $\struct{X}$ is
  isomorphic to $\sheaf{S}(A)$.  Then for every finitely generated projective
  $A$-module $P$, the sheaf $\sheaf{S}(P)$ is a locally free
  $\struct{X}$-module of bounded rank.
\end{lemma}
\begin{proof} \label{proof:6} It is given that, an $A$-module $P$ is finitely
  generated and projective, hence it is of finite presentation.  Therefore, we
  get an exact sequence of $A$-modules $A^p \rightarrow A^q \rightarrow P
  \rightarrow 0$, for some $p, q \in \N$.  Since the functor $\sheaf{S}$ is
  right exact, and by hypothesis $\sheaf{S}(A^m) \iso \struct{X}^m$ for all $m
  \in \N$, we get an exact sequence of $\struct{X}$-modules $\struct{X}^p
  \rightarrow \struct{X}^q \rightarrow \sheaf{S}(P) \rightarrow 0$.  This
  shows that $\sheaf{S}(P)$ is of finite presentation.  Since for every $x \in
  X$, $\sstalk{X}{x}$ is a local ring, and $P$ is a finitely generated
  projective module $\Tensor[A]{P}{\sstalk{X}{x}}$ is a free
  $\sstalk{X}{x}$-module of finite rank.  We denote the rank
  $\rk[x]{\sheaf{S}(P)}$ of the sheaf $\sheaf{S}(P)$ at every point $x$ in $X$
  by $n_x$.  Therefore $\Tensor[A]{P}{\sstalk{X}{x}}$ is isomorphic to
  $\nsstalk[n_x]{X}{x}$.  Now by Lemma \ref{lemma:2} $\sheaf{S}(P)$ is a
  locally free $\struct{X}$-module.  Also the family of integers $(n_x)_{x\in
    X}$ is bounded above by $q$, so the sheaf $\sheaf{S}(P)$ is of bounded
  rank.
\end{proof}

Now we are ready for giving proof of the main theorem.

\begin{proof}[{\bf Proof of Theorem \ref{SS Theorem}}]\label{proof:7} 
  It follows from Proposition \ref{prop:3} that, if an $\struct{X}$-module
  $\sheaf{F}$ is locally free of bounded rank then $\stion{\sheaf{F}}{X}$ is
  finitely generated projective $A$-module, hence the restriction of the
  functor $\stion{\dummy}{X}$ from the subcategory $\Lfb{X}$ to the
  subcategory $\Fgp{A}$ is welldefined.  Since $\Lfb{X}$ is a subcategory of
  the admissible subcategory $\cat{C}$, Proposition \ref{prop:1} implies that,
  $\stion{\dummy}{X}$ is fully faithful.  By Corollary \ref{cor:1}, if $P$ is
  finitely generated projective module, then $P$ is isomorphic to
  $\stion{\sheaf{S}(P)}{X}$.  The sheaf $\sheaf{S}(P)$ is locally free of
  bounded rank is follows from Proposition \ref{prop:2} and Lemma
  \ref{lemma:3}.  Hence the functor
  $\mor{\stion{\dummy}{X}}{\Lfb{X}}{\Fgp{A}}$ is essentially surjective.
  Therefore, $\stion{\dummy}{X}$ is an equivalence of categories.  Moreover,
  $\sheaf{S}$ is a quasi-inverse of $\stion{\dummy}{X}$ is a consequence of
  Proposition \ref{prop:2} and Lemma \ref{lemma:3}.
\end{proof}

\section{Some Special Cases}
\label{sec:special cases}

In this section we will discuss some important examples of locally ringed
spaces for which the Serre-Swan Theorem holds.

Let $\ringed{X}$ be an affine scheme, with a coordinate ring $A$.  Let
$\mathcal{B}$ denote the base for the topology on $X$, which consist of the
principal open sets $\X{f}$ $(f\in A)$.  For every $A$-module $M$ let
$\tilde{M}$ denote the $\struct{X}$-module associated to $M$ \cite[Chap. I,
D\'efinition (1.3.4), p.~85]{Gro60}.  (Recall that $\tilde{M}$ is defined by
$\stion{\tilde{M}}{\X{f}}=\loc{M}{f}$.)  Let $\Qcoh{X}$ denote the full
subcategory of $\struct{X}$-module consisting of quasicoherent
$\struct{X}$-modules.  The functor $\tilde{\dummy}$ is an equivalence of
categories from $\modcat{A}$ to $\Qcoh{X}$, with quasi-inverse
$\mor{\stion{\dummy}{X}}{\Qcoh{X}}{\modcat{A}}$ \cite[Chap. I, Th\'eor\`eme
(1.4.1), p.~90]{Gro60}.  The category $\Qcoh{X}$ clearly satisfies condition
{\bf C1} of Definition \ref{def:1}.  It is easy to check that the functor
$\sheaf{S}$ and the functor $\tilde{\dummy}$ are isomorphic.
\begin{corollary} \label{Serre's Theorem}{\bf \sc (Serre's Theorem)}{\rm
    \cite[Section 50, Corollaire to Proposition 4, p.~242]{Ser55}} Let
  $\ringed{X}$ be an affine scheme, and let $A$ denote its coordinate ring
  $\stion{\struct{X}}{X}$.  Then, a sheaf $\sheaf{F}$ is locally free
  $\struct{X}$-module of finite rank if and only if $\stion{\sheaf{F}}{X}$ is
  a finitely generated projective $A$-module.  The functor
  $\mor{\stion{\dummy}{X}}{\Lfb{X}}{\Fgp{A}}$ is an equivalence of categories,
  with a quasi-inverse $\mor{\tilde{\dummy}}{\Fgp{A}}{\Lfb{X}}$.
\end{corollary}
\begin{proof} \label{proof:8} Quasicoherent $\struct{X}$-modules over an
  affine scheme are acyclic \cite[Theorem 2.18, p.~186]{Liu2002}.  Let sheaf
  $\sheaf{F}$ be an quasicoherent sheaf, and $M=\stion{\sheaf{F}}{X}$.  Then
  $\sheaf{F} \iso \tilde{M}$, and $\tilde{M}$ is clearly generated by global
  sections.  Obvious $\Lfb{X}$ is a subcategory of $\Qcoh{X}$, hence
  $\Qcoh{X}$ is an admissible subcategory.  Since $X$ is quasicompact, locally
  free sheaves of finite rank are finitely generated by global sections.
  Therefore the corollary will follows from Theorem \ref{SS Theorem}.
\end{proof}

Consider a ringed space $\ringed{X}$ as in Lemma \ref{lemma:1}.
Further, assume that $X$ is of bounded topological dimension.  Then,
locally free sheaves of bounded rank over $X$ are finitely generated
by global sections.  This fact is standard when $\ringed{X}$ is a
differential manifold \cite[Chap. III, Proposition 4.1]{Wel80}, and the
proof in the general case is similar.

\begin{corollary} \label{cor:2} 
  Let $\ringed{X}$ be a ringed space such that, $X$ is a paracompact
  topological space of bounded topological dimension, and $\struct{X}$ is a
  fine sheaf.  Then, the Serre-Swan Theorem holds for $\ringed{X}$.
\end{corollary}
\begin{proof} \label{proof:9} The category $\modcat{\struct{X}}$ clearly
  satisfies {\bf C1} and {\bf C3} of Definition \ref{def:1}.  Since $X$ is a
  paracompact topological space, fine sheaves on $X$ are soft, and hence
  acyclic.  Since $\struct{X}$ is a fine sheaf, every $\struct{X}$-module is
  fine, and hence acyclic.  Also by Lemma \ref{lemma:1} $\sheaf{F}$ is
  generated by global sections, hence $\modcat{\struct{X}}$ satisfies {\bf
    C2}.  From the previous paragraph every sheaf $\sheaf{F}$ in $\Lfb{X}$ is
  finitely generated by global sections.  Now the corollary follows applying
  Theorem \ref{SS Theorem}, with $\cat{C}=\modcat{\struct{X}}$.
\end{proof}

The sheaf of continuous real-valued functions on a paracompact topological
space is a fine sheaf.  Hence, the following is an immediate consequence of
Corollary \ref{cor:2}.

\begin{corollary} \label{Swan's Theorem}{\bf \sc (Swan's Theorem)}{\rm
    \cite[Theorem 2 and p.~277]{Sw62}} Let $X$ be a paracompact topological
  space of bounded topological dimension, and let $\contstruct{X}$ denote the
  sheaf of continuous real-valued functions on $X$.  Let $\contralg{X}$ denote
  the $\R$-algebra $\stion{\contstruct{X}}{X}$.  Then, the functor
  $\mor{\stion{\dummy}{X}}{\Lfb{X}}{\Fgp{\contralg{X}}}$ is an equivalence of
  categories.
\end{corollary}

It follows from Corollary \ref{cor:2} that if $X$ is a differentiable manifold
of bounded dimension, and $\cinfty[\R]$ (respectively $\cinfty[\C]$) is the
sheaf of differentiable real-valued (respectively, complex-valued) functions
on $X$, then the category of real (respectively, complex) differentiable
vector bundles on a manifold $X$ is equivalent to the category of finitely
generated projective $\stion{\cinfty[\R]}{X}$-modules (respectively
$\stion{\cinfty[\C]}{X}$-modules).  Another interesting example is that of
affine differentiable spaces.
\begin{corollary} \label{cor:3}\cite[Theorem 3.11; Theorem
    4.16]{GS2003} Let $A$ be a differentiable algebra, and let $\ringed{X}$
  denote its real spectrum $\adiffspace{A}$ {\rm(\cite[pp.~22, 30, 44]{GS2003})}.
  Then, the functor,
  $\mor{\stion{\dummy}{X}}{\modcat{\struct{X}}}{{\modcat{A}}_{\modcat{\struct{X}}}}$
  is an equivalence of categories.  Moreover, the Serre-Swan Theorem holds for
  $X$.
\end{corollary}
\begin{proof}\label{proof:10}
  The topological space $X$ is homeomorphic to a closed subset of $\R^n$ for
  some finite $n$ \cite[Proposition 2.13]{GS2003}, therefore
  $X$ is a paracompact topological space of bounded topological dimension.
  Since the sheaf $\struct{X}$ admits partition of unity \cite[Theorem of
  partition of unity, p.~52]{GS2003} it a fine sheaf.  Now the corollary
  follows from Remark \ref{remark:1} and Corollary \ref{cor:2}.
\end{proof}

It is easy to see that the functor $\sheaf{S}$ is isomorphic to the functor
$\tilde{\dummy}$ defined as in \cite[Definition, p.~40]{GS2003}.

In the paper \cite{Mul76}, Mulvey proved that for a locally ringed space
\ringed{X} (not necessarily commutative) whose center is compact, the
Serre-Swan theorem holds.  A ringed space $\ringed{X}$ is said to be compact
provided that, the topological space $X$ is compact, and that for every $x, x'
\in X$, there exists an element $a\in \stion{\struct{X}}{X}$ satisfying
$a(x)=1$ and $a(x')=0$, \cite[Section 3, definition, p.~63]{Mul76}.  If
$\ringed{X}$ is commutative (i.e., $\stion{\struct{X}}{X}$ is a commutative
ring) then the above result follows from Theorem \ref{SS Theorem}.

\begin{corollary}\label{cor:4}{\rm \cite[Theorem 4.1]{Mul76}} 
  If a locally ringed space $\ringed{X}$ is compact, then the category of
  locally free $\struct{X}$-modules of bounded rank is equivalent to the
  category of finitely generated projective $\stion{\struct{X}}{X}$-modules.
\end{corollary}
\begin{proof} \label{proof:11} The structure sheaf $\struct{X}$ is a fine
  sheaf (this is follows from \cite[Corollary 1.3]{Mul78}), and $X$ is a
  compact topological space.  Therefore, by Lemma \ref{lemma:1} every sheaf in
  $\modcat{\struct{X}}$ is generated by global sections.  Also, since $X$ is
  compact, locally free sheaves of bounded rank over $X$ are finitely
  generated by global sections.  On a paracompact space fine sheaves are
  acyclic, hence all $\struct{X}$-modules are acyclic.  Therefore the category
  $\modcat{\struct{X}}$ is admissible.  Now, the corollary follows from
  Theorem \ref{SS Theorem} by taking category $\cat{C}$ to be the category
  $\modcat{\struct{X}}$.
\end{proof}

Recall that \cite[p.~ 8, Definition 10.2]{Pie67}, a ringed space $\ringed{X}$
is called {\it regular ringed space} if $X$ is a profinite space, i.e., a
compact totally disconnected space, and $\sstalk{X}{x}$ is a field for
every $x\in X$.
\begin{corollary} \label{cor:5}{\rm \cite[Theorem 15.3]{Pie67}} Let
  $\ringed{X}$ be a regular ringed space, and let $A$ denote
  $\stion{\struct{X}}{X}$. Then Serre-Swan Theorem holds for $\ringed{X}$.
  Moreover, coherent sheaves over $X$ are locally free $\struct{X}$-modules of
  bounded rank.  Hence, $\Coh{X}$, the category of coherent sheaves over $X$, 
  and $\Fgp{A}$ are equivalent.
\end{corollary}
\begin{proof}\label{proof:12} 
  Note that regular ringed spaces are commutative and compact ringed spaces
  \cite[pp.~65-66]{Mul76}.  Therefore, the Serre-Swan Theorem holds for
  $\ringed{X}$ by Corollary \ref{cor:4}.  Let $\sheaf{F}$ be a coherent
  $\struct{X}$-module.  Since $X$ is a regular ringed space, $\sstalk{X}{x}$
  is a field for every $x\in X$.  Thus, $\stalk{\sheaf{F}}{x}$ is a free
  $\sstalk{X}{x}$-module for every $x\in X$.  Since coherent sheaves are of
  finitely presentation, Lemma \ref{lemma:2} implies that $\sheaf{F}$ is
  locally free.  Also $X$ is compact, therefore $\sheaf{F}$ is of bounded
  rank.  Hence, $\sheaf{F}$ is a locally free $\struct{X}$-module of bounded
  rank.
\end{proof}

Let $\ringed{X}$ be a Stein space, and let $A=\stion{\struct{X}}{X}$.  Recall
that, a topological module $M$ over the algebra $A$ is called a {\it Stein
  module} if there exists a coherent $\struct{X}$-module $\sheaf{M}$, such
that $\stion{\sheaf{M}}{X}$ is isomorphic to $M$ \cite[Section 2,
p.~383]{For67}.  Let $\modcat{\Stein{S}}$ denote the category of Stein modules
over $A$.  Since $\struct{X}$ is coherent, $\Coh{X}$ satisfy {\bf C1} and {\bf
  C3} of Definition \ref{def:1}.  Recall that, {\sc Theorem A} and {\sc
  Theorem B} are valid for Stein spaces \cite[Chapter IV, Section 1, Theorem
2, p.~101 and Chapter V, Section 4, Theorem 3, p.~152]{GR79}.  Therefore, the
category $\Coh{X}$ satisfies {\bf C2}.  Thus, $\Coh{X}$ is an admissible
subcategory of $\modcat{\struct{X}}$.  Note that the full subcategory
${\modcat{A}}_{\Coh{X}}$ of category $\modcat{A}$ is precisely
$\modcat{\Stein{S}}$.
\begin{corollary} \label{Forster's theorem}{\rm \cite[Satz 6.7 and Satz
    6.8]{For67}}
  Let $\ringed{X}$ be a finite-dimensional connected Stein space.  Then
  $\mor{\stion{\dummy}{X}}{\Coh{X}}{\modcat{\Stein{S}}}$ is an equivalence of
  category with quasi-inverse $\mor{\sheaf{S}}{\modcat{\Stein{S}}}{\Coh{X}}$.
  Moreover, the category of locally free sheaves of finite rank is equivalent
  to the category of finitely generated projective $A$-modules.
\end{corollary}
\begin{proof} \label{proof:13} Let $A=\stion{\struct{X}}{X}$.  Since $\Coh{X}$
  is an admissible subcategory, by Theorem \ref{SS Theorem} and Remark
  \ref{remark:1} to prove the corollary it is enough to prove that, every
  locally free sheaf of bounded rank is finitely generated by global sections.
  Let $\sheaf{M}$ be locally free sheaf on $X$ of bounded rank, and
  $M=\stion{\sheaf{M}}{X}$ be the corresponding Stein module over $A$.  For a
  Stein module $M$ let $d_x(M)$ denote the rank of $\sheaf{M}$ at a point $x$.
  Let $d=\sup\set{d_x(M)\suchthat \forall x \in X}$.  Since $\sheaf{M}$ is of
  finite rank $d<\infty$.  Since $X$ is connected and finite dimensional, $A$
  is indecomposable and finite dimensional \cite[Section 1, Subsection 3,
  p.~382]{For67}.  Therefore, $M$ is a finitely generated $A$-module
  \cite[Corollary 4.7]{For67}.  Now since $\sheaf{M}$ is generated by global
  sections, and $\stion{\sheaf{M}}{X}$ is a finitely generated, $\sheaf{M}$ is
  finitely generated by global sections.
\end{proof}

Note that the spectrum of $A$, denoted by $S(A)$ is homeomorphic to $X$
\cite[Section 1, Satz 1, Beweis c, p.~380]{For67}.  Therefore the definition
of $d$ in \cite[Corollary 4.7]{For67} is same as that in the Corollary
\ref{Forster's theorem}.

\vspace{.3cm}

{\small {\it Acknowledgment.}  The author would like to thank the
  anonymous referee for helpful comments and suggestions.}

\end{document}